\documentclass[a4paper, 10pt, nosumlimits, twoside] {amsart}
\usepackage{amsmath} 
\usepackage{amssymb}
\usepackage{amsfonts}
\usepackage{amsthm}
\usepackage{mathrsfs}
\usepackage{graphics,color}
\usepackage{epsfig}
\usepackage{enumerate}
\usepackage[T1]{fontenc}
\usepackage[arrow, matrix, curve]{xy}
\usepackage[pdfpagemode=UseNone,
						pdfstartview=FitH,
						hyperfigures=true,
						hyperfootnotes=true,
						hypertexnames=true,
						pdftitle={The Cone of Moving Curves on Algebraic Varieties},
						pdfauthor={Sammy Barkowski},
						pdfkeywords={moving, cone, Mori, minimal model program, mmp,
												 curve, divisor, numerical pullback,
												 numerical pushforward}
						]{hyperref}
\hypersetup{
   					colorlinks=true,
    				linkcolor=black,
    				citecolor=black,
    				filecolor=black,
    				urlcolor=black,
					 }
\numberwithin{subsection}{section}
\newtheoremstyle{note}
  {}
  {}
  {}
  {}
  {\bfseries}
  {.}
  { }
  {}

\newtheoremstyle{notes}
  {}
  {}
  {\itshape}
  {}
  {\bfseries}
  {.}
  { }
  {}
\theoremstyle{notes}
\newtheorem{satz}{Satz}[section]
\newtheorem{lemma}[satz]{Lemma}
\newtheorem{thm}[satz]{Theorem}
\newtheorem{prop}[satz]{Proposition}
\newtheorem{cor}[satz]{Corollary}
\theoremstyle{note}
\newtheorem{bem}[satz]{Remark}
\newtheorem{fact}[satz]{Fact}
\newtheorem{defn}[satz]{Definition}
\newtheorem{bsp}[satz]{Example}
\newtheorem{constr}[satz]{Construction and Definition}
\newtheorem{summ}[satz]{Summary}
\theoremstyle{remark}
\newtheorem{claim}[satz]{Claim}
\renewenvironment{proof}[1][\proofname]%
  {\par\addvspace{.6pc plus .2pc minus .1pc}
   \noindent
   {\itshape{#1.}}\enspace\ignorespaces}%
   {\qed}
  {\par\addvspace{.6pc plus .2pc minus .1pc}}
\def\proofname{Proof}

\newcommand{\N}{\mathbb{N}}                            
\newcommand{\Q}{\mathbb{Q}}														 
\newcommand{\R}{\mathbb{R}}                            
\newcommand{\pro}{\mathbb{P}}													 
\newcommand{\Nd}[1]{N^1(#1)_{\mathbb{R}}}							 
\newcommand{\Nc}[1]{N_1(#1)_{\mathbb{R}}}							 
\newcommand{\mor}[1]{\overline{\textup{NE}}({#1})}     
\newcommand{\mov}[1]{\overline{\textup{NM}}({#1})}     
\newcommand{\nio}[1]{\textup{NI}({#1})}								 
\newcommand{\nic}[1]{\overline{\textup{NI}}({#1})}		 
\newcommand{\eff}[1]{\overline{\textup{Eff}}({#1})}    
\newcommand{\nef}[1]{\textup{Nef}({#1})}							 
\newcommand{\eq}[1]{\textup{Eq}({#1})}								 
\newcommand{\polycone}{pmc-fourfold}						 			 
\newcommand{\pmc}{pmc-flip sequence}
\newcommand{\sE}{\mathscr{E}}     
\newcommand{\cO}{\mathcal{O}}     

\newcommand{\Ra}{\Rightarrow}			 
\newcommand{\ra}{\rightarrow}      
\newcommand{\bir}{\dashrightarrow} 
\newcommand{\xra}{\xrightarrow}    

\DeclareMathOperator{\ldot}{_{\cdot}}        
\DeclareMathOperator{\iso}{\cong}            
\newcommand{\nprop}{\sim_{\text{num}}}       

\newcommand{\eps}{\varepsilon}
\newcommand{\ph}{\varphi}
\newcommand{\nb}{\varphi_1^*}     
\newcommand{\np}{\varphi_{*1}}    
\newcommand{\npb}[1]{{#1}{_1^*}}  
\newcommand{\npf}[1]{{#1}{_{*1}}} 

\begin{document}

\title{The cone of moving curves of a smooth Fano three- or fourfold}

\author{Sammy Barkowski}
  
\address{Graduiertenkolleg Globale Strukturen, Mathematisches Institut der Universit\"at zu
  K\"oln, Weyertal 86--90, 50931 K\"oln, Germany}
\email{\href{mailto:SamBarkowski@googlemail.com}{SamBarkowski@googlemail.com}}

\date{February 02, 2009}

\begin{abstract}
We describe the closed cone of moving curves $\mov{X}\subset\Nc{X}$ of a smooth Fano three- or fourfold $X$ by finitely many linear equations. These equations are induced by the exceptional divisors of divisorial contractions and nef divisors on birational models of $X$ which are obtained by flips. The proof provides an inductive way to compute the cone $\mov{X}$ of moving curves and gives a description of the Mori cone of a variety $X^+$ obtained by a flip $\phi:X\bir X^+$ of a small contraction on $X$.
\end{abstract}
\maketitle

\setcounter{tocdepth}{1}

\tableofcontents
\section{Introduction}
We will give a characterisation of the cone $\mov{X}$ of moving curves of a smooth Fano three- or fourfold $X$ by a finite number of linear equations, which are related to exceptional divisors of divisorial contractions and nef divisors on birational models of $X$:
\begin{thm}[Main Theorem]\label{thm:mainthm}
Let $X$ be a smooth Fano fourfold. Then there exists an explicitely known, finite set $\eq{X}\subset\Nd{X}$ of divisor classes such that  \[\mov{X}=\{\gamma\in\Nc{X}\mid \gamma\ldot\Delta\geq0\text{ for all }\Delta\in\eq{X}\}.\]
In particular, $\mov{X}$ is a closed convex polyhedral cone in $\Nc{X}$.
\end{thm}
We will construct the set $\eq{X}$ explicitely in the following sections, and we will show by example that this set is actually computable.

The proof of the main theorem also applies to smooth Fano threefolds. We will thus deduce the following statement.
\begin{prop}\label{prop:mainprop}
Let $X$ be a smooth Fano threefold and let $\ph_i:X\ra X_i$ be the divisorial contractions of $X$, with exceptional divisors $E_i\subset X$ and extremal rays $\R_+[r_i]$, $i=1,\ldots,k$, of the Mori cone $\mor{X}$. Then
\[\mov{X}=\{\gamma\in\Nc{X}\mid \gamma\ldot\Delta\geq0\text{ for all }\Delta\in\left\{[E_1],\ldots,[E_k]\right\}\cup\nef{X}\}.\]
In particular, $\mov{X}$ is a closed convex polyhedral cone in $\Nc{X}$.
\end{prop}

There are several authors who have been working on the cone of moving curves. As a consequence of their main theorem in the preprint \cite{bchm}, C. Birkar, P. Cascini, C. Hacon and J. M$^{\textup{c}}$Kernan obtain that the cone of moving curves of a Fano $n$-fold is polyhedral.

In contrast to our approach, they describe the cone in terms of extremal rays which can be obtained as pullbacks of rational curves lying in a general fibre of a Mori fibre space obtained by running the minimal model program. 

This idea was introduced in a more general setup in the paper \cite{Bat} by V. V. Batyrev. There he gives a structure theorem for the moving cone of a threefold, which was finally proved in the preprint \cite{Arau} by C. Araujo, but see \cite[Remark~3.4]{Arau}. Very recently, B. Lehmann published the preprint \cite{leh}. He proves a version of this structure theorem for higher-dimensional varieties. The author has been informed that Alex K\"uronya and Endre Szab\'o have independently obtained a description of the moving curves of manifolds in dimension three and four.

The proof of Theorem~\ref{thm:mainthm} and Proposition~\ref{prop:mainprop} rely on the famous result \cite[Theorem~2.2]{bdpp} of Bucksom, Demailly, Paun and Peternell. In the case of a smooth Fano fourfold we additionally employ a theorem \cite[Theorem~1.1]{kaw} of Y. Kawamata that describes flips in detail.
\subsection*{Acknowledgments}

The results presented here are part of the author's Ph.D. thesis \cite{diss}, 
written under the supervision of Stefan Kebekus. The author would like to 
thank the Graduiertenkolleg ``Globale Strukturen in Geometrie und Analysis'' of the Deutsche Forschungsgemeinschaft, DFG, which granted him a full scholarship. He would also like to thank James \hbox{M$^{\textup{c}}$Kernan}, Se\'an Keel, Alex K\"uronya, Carolina Araujo, Laurent Bonavero and particularly Cinzia Casagrande for answering his questions. Moreover, he is thankful to Laurent Bonavero and Andreas H\"oring for some inspiring discussions during the summer school ``Geometry of complex projective varieties and the minimal model program'' in 2007 and a workshop in 2008 at the Institut Fourier in Grenoble.
\section{Flips for smooth fourfolds}\label{moving}
\subsection{Essential results}
In this short section we will cite the results which will be the technical basis for the proof of Theorem~\ref{thm:mainthm}.
The first theorem allows us to describe the moving cone in terms of linear equations.
\begin{thm}[see \protect{\cite[Theorem~2.2~and~Theorem~2.4]{bdpp}}]\label{thm:eq}
Let $X$ be an irreducible projective variety of dimension $n$. Then the cone $\mov{X}$ of moving curves and the cone $\eff{X}$ of pseudoeffective divisors are dual; in other words, \[\mov{X}=\{\gamma\in\Nc{X}\mid \gamma\cdot\Delta\geq 0,\text{ for all }\Delta\in\eff{X}\}.\]\qed
\end{thm}
The reader who is not familiar with the subject may take this statement as the definition for the moving cone. For a detailed treatment see \cite[Chapter~1]{diss}.

If the Mori cone $\mor{X}$ of a smooth Fano fourfold $X$ has a small extremal ray, we have to use the flip of the corresponding small contraction to obtain some of the linear equations which cut out the moving cone $\mov{X}$ of $X$. The following result, which is due to Kawamata, gives a precise description of the exceptional locus of a small contraction on a smooth fourfold. 
\begin{thm}[\protect{\cite[Theorem~1.1]{kaw}}]\label{thm:kawamata}
Let $X$ be a non-singular projective variety of dimension four and let $\ph:X\ra Y$ be a small contraction. Then the exceptional locus $S$ of $\ph$ is a disjoint union of its irreducible components $S_i$, $i=1\ldots n$, such that $S_i\iso\pro^2$ and $N_{S_i/X}\iso\cO_{\pro^2}(-1)\oplus\cO_{\pro^2}(-1)$, where $N_{S_i/X}$ denotes the normal bundle of $S_i$ in $X$.\qed
\end{thm}
Moreover, Kawamata proves the following corollary.
\begin{cor}[\protect{\cite[Corollary~1.2]{kaw}}]\label{cor:kawamata}
Let $\ph:X\ra Y$ be as in Theorem~\ref{thm:kawamata}. Then there exists a flip
\begin{equation}\label{eq:flipdiagram}
  		\xymatrix{
      					X \ar@{-->}[rr]^{\phi} \ar[rd]_{\ph}  &     &  X^+ \ar[dl]^{\ph^+}  \\
                             													&  Y  &
  						 }
\end{equation}
where $X^+$ is a non-singular projective variety.\qed
\end{cor}
\begin{bem}\label{bem:kawamata}
The flip is constructed as follows. If we blow up the exceptional locus $S$ in $X$, then the exceptional divisor of the blow-up is a disjoint union of irreducible components $E_i\iso\pro^2\times\pro^1$. Furthermore, the normal bundle of each $E_i$ is isomorphic to $pr_1^*(\cO_{\pro^2}(-1))\oplus pr_2^*(\cO_{\pro^1}(-1))$, where $pr_1$, resp.~$pr_2$, denotes the projection on the first, resp.~second, factor of $\pro^2\times\pro^1$. By contracting the exceptional divisor in the other direction, we obtain a smooth projective variety $X^+$ and the commutative flip-diagram~\eqref{eq:flipdiagram}, but see \cite[Corollary~1.2]{kaw}.
\end{bem}
\begin{bem}[\protect{\cite[Remark~4.18]{diss}}]\label{rmk:negpb}
Let $X$ be as in Theorem~\ref{thm:kawamata} with flip diagram~\eqref{eq:flipdiagram}. Moreover, let $\gamma$ be the class of a line $g$ which lies in a fibre of $\ph$ and let $\gamma^+$ be the class of a curve $g^+$ which lies in a fibre of $\ph^+$.
\begin{fact} The curve $g^+$ is isomorphic to $\pro^1$, and an elementary computation using the normal bundle sequence shows that \[N_{g/X}\iso\cO_g(1)\oplus(\cO_g(-1)^{\oplus 2})\text{ and }N_{g^+/X^+}\iso\cO_{g^+}(-1)^{\oplus 3}.\] Taking the degree of the first Chern classes in the normal bundle sequences for $N_{g/X}$ and $N_{g^+/X^+}$ yields that \begin{equation}\label{eq:number_one}\gamma^+\ldot[K_{X^+}]=1=-\gamma\ldot[K_X].\end{equation} This equation will be very useful in the proof of the main theorem.
\end{fact}
\end{bem}
We will now introduce a method to take pullbacks or pushforwards of $1$-cycles via a flip, which will be used throughout the whole paper. For more details see \cite[Section~3]{Arau} or \cite[Chapter~3]{diss}.
\begin{fact}[\protect{\cite[Definition~3.1]{Arau}}]\label{fact:fact}
Let $\ph:X\bir Y$ be a birational map between smooth projective varieties which is an isomorphism in codimension one. The pullback of $\Q$-Cartier divisors on $Y$ via $\ph$ yields an injective linear map \[\ph^*:\Nd{Y}\hookrightarrow\Nd{X},\] 
between the Néron-Severi vector spaces of numerical equivalence classes of $\R$-divisors on $Y$, and on $X$. Analogously, the pushforward of $\Q$-Cartier divisors on $X$ via $\ph$ yields a surjective linear map \[\ph_*:\Nd{X}\twoheadrightarrow\Nd{Y}\]
such that $\ph_*\circ\ph^*=\textup{id}_{\Nd{Y}}$.
\end{fact}
\begin{defn}
Let \[\nb:\Nc{Y}\hookrightarrow\Nc{X}\] be the dual linear map of the pushforward $\ph_*:\Nd{X}\twoheadrightarrow\Nd{Y}$ of $\R$-divisors on $X$ and let \[\np:\Nc{X}\twoheadrightarrow\Nc{Y}\] be the dual linear map of the pullback $\ph^*:\Nd{Y}\hookrightarrow\Nd{X}$ of $\R$-divisors on $Y$. We call $\nb$ the \emph{numerical pullback via} $\ph$ and $\np$ the \emph{numerical pushforward via} $\ph$.
\end{defn}
The numerical pullback and the numerical pushforward satisfy some useful properties, as a projection formula, by definition.

The following is immediate from the definition. For a detailed proof, see \cite[Chapter~4.2.2]{diss}.
\begin{lemma}\label{lem:explicit}
Let $X$ be a smooth fourfold and let $\ph:X\ra Y$ be a small contraction with flip diagram \eqref{eq:flipdiagram}.
\begin{enumerate}[(i)]
\item If $\gamma$ is the class of a line which lies in a fibre of $\ph$ and $\gamma^+$ is the class of a curve which lies in a fibre of $\ph^+$, then 
\[
\npb{\phi}(\gamma^+)=-\gamma.
\]
\item Let $c^+$ be an irreducible curve on $X^+$. Then $\npb{\phi}([c^+])$ is an effective class if and only if $c^+$ is not contained in the exceptional locus $S^+$ of the flipped small contraction $\ph^+$. More precisely, if $c^+$ is not contained in $S^+$, then \[\npb{\phi}([c^+])=[c]+k[r_s],\] 
where $c$ is the strict transform of $c^+$, $k\geq0$ with equality iff $c^+$ is disjoint from $S^+$ and $r_s$ is an irreducible curve which is contained in a fibre of the small contraction $\ph$. The analogous statement holds for the numerical pushforward of a curve which is not contracted by $\ph$.
\end{enumerate}\qed
\end{lemma}
\subsection{The Mori cone of a smooth fourfold}
The following proposition will later serve as an induction step in the proof of the fact that the moving cone of a smooth fourfold is polyhedral. The proof of the proposition relies on Kawamata's results. It essentially uses Lemma~\ref{lem:explicit} and the fact that the exceptional locus of the flipped small contraction is a disjoint union of finitely many rational curves.
\begin{prop}\label{prop:inductionstep}
Let $X$ be a smooth projective fourfold such that $K_X$ fails to be nef and such that $\mor{X}$ is a convex, polyhedral cone in $\Nc{X}$, \[\mor{X}=\langle[p_1],\ldots,[p_k],[n_1],\ldots,[n_m]\rangle_{\R_+}\text{ say,}\] where each $p_i$ is an irreducible curve on $X$ such that $[p_i]\ldot[K_X]\geq0$. Assume that there exists an ample divisor $A$ on $X$ and a real number $\eps>0$ such that
\begin{enumerate}[(i)]
\item every irreducible curve $c\neq p_1,\ldots,p_k$ on $X$ is $(K_X+\eps A)$--negative\label{prop:assumption2} and that
\item each $n_j$ is a rational curve on $X$, in particular $[n_j]\ldot([K_X]+\eps[A])<0$ by (\ref{prop:assumption2}).\label{prop:assumption3}
\end{enumerate}
Moreover, assume that $[n_m]$ is a small extremal class with extremal contraction $\ph:X\ra Y$ and let
\[
  \xymatrix{
      X \ar@{-->}[rr]^{\phi} \ar[rd]_{\ph}  &     &  X^+ \ar[dl]^{\ph^+}  \\
                             &  Y  &
  }
\]
be the flip of $\ph$. Denote by $S:=\textup{Exc}_X(\ph)$ the exceptional set of $\ph$ in $X$ and by $S^+:=\textup{Exc}_{X^+}(\ph^+)$ the exceptional set of $\ph^+$ in $X^+$.

Then $\mor{X^+}$ is a convex, polyhedral cone in $\Nc{X^+}$. For every ample divisor $A^+$ on  $X^+$ there exists an $\eps^+>0$ such that every irreducible curve $c^+\neq p_1^+,\ldots,p_k^+$ which is not contained in $S^+$ is $(K_{X^+}+\eps^+A^+)$--negative, where $p_i^+$ denotes the strict transform of $p_i$ under $\phi$, for $i=1,\ldots,k$.
\end{prop}
\emph{Strategy for the proof.} Since $\mor{X^+}$ contains no lines, it is the span of its extremal rays.
Therefore, to show that $\mor{X^+}$ is polyhedral, it is sufficient to show that $\mor{X^+}$ has only finitely many extremal rays.
There exists a decomposition \[\mor{X^+}=\nic{X^+}+\R_+[n_m^+]+\R_+[p_1^+]+\dots+\R_+[p_k^+],\] where $\nic{X^+}$ is a certain subcone of $\mor{X^+}$ that we will define later.
Any extremal ray $R$ of $\mor{X}$ which is not equal to $\R_+[n_m^+]$ or $\R_+[p_i^+]$, $i=1,\ldots,k$, is then an extremal ray of $\nic{X^+}$.
We will then show that for every ample divisor $A^+$ on $X^+$ there exists an $\eps^+\in\R_{>0}$ such that $\nic{X^+}\setminus\{0\}$ is entirely $[K_{X^+}+\eps^+A^+]$--negative.
Therefore, every extremal ray $R$ of $\mor{X}$ which is not equal to $\R_+[n_m^+]$ or $\R_+[p_1^+],\ldots,\R_+[p_k^+]$ is $[K_{X^+}+\eps^+A^+]$--negative. This will conclude the proof since Mori's Cone Theorem says that $\mor{X^+}$ has only finitely many $[K_{X^+}+\eps^+A^+]$--negative extremal rays.
\begin{proof}
First of all, note that the existence of the map $\phi$ is guaranteed by Corollary~\ref{cor:kawamata}, and note that $X^+$ is a smooth fourfold. Moreover, we know that $S$ is the disjoint union of its irreducible components, which are isomorphic to $\pro^2$, and that $S^+$ is a disjoint union of smooth rational curves.

Without loss of generality, we may assume that $S$ and $S^+$ are irreducible. For conformity, we will denote the irreducible curve $S^+$ by $n_m^+$. Furthermore, we can assume that \[\npb{\phi}([n_m^+])=-[n_m]\] by Lemma~\ref{lem:explicit}. Recall that for every irreducible curve $r^+\neq n_m^+$ on $X^+$ the class $\npb{\phi}([r^+])$ is effective and not zero by Lemma~\ref{lem:explicit}. 
Let $p_i^+$ be the strict transform of $p_i$ under $\phi$ for all $i=1,\ldots,k$, define $P:=\{n_m^+,p_1^+,\ldots,p_k^+\}$,
\[\nio{X^+}:=\bigl\{\sum\limits_{i=1}^sa_i[c_i^+]\mid a_i\in\R_+, c_i^+\notin P\text{ is an irreducible curve on }X^+\bigr\}.\]
and denote by $\nic{X^+}$ the closure of $\nio{X^+}$ in $\Nc{X^+}$. A straightforward computation shows that
\[\mor{X^+}=\nic{X^+}+\R_+[n_m^+]+\R_+[p_1^+]+\dots+\R_+[p_k^+].\]
By Assumption~\eqref{prop:assumption2}, we know that all irreducible curves on $X$ except $p_1,\ldots,p_k$ are $(K_X+\eps A)$--negative for a certain ample divisor $A$ on $X$ and a certain real number $\eps>0$. 

Now let $A^+$ be an arbitrary ample divisor on $X^+$ and set 
\[\eps':=\min\left\{\tfrac{\eps[A]\ldot[p_i]}{\phi^*([A^+])\ldot[p_i]},\tfrac{\eps[A]\ldot[n_j]}{\phi^*([A^+])\ldot[n_j]}\mid i=1,\ldots,k,\ j=1,\ldots,m-1\right\}.\]
We have to check that $\eps'$ is well-defined. Since none of the $p_i$, $1\leq i\leq k$, or $n_j$, $1\leq j < m$, are contained in the exceptional locus $S$ of $\ph$, all the classes $\npf{\phi}([p_i])$ and $\npf{\phi}([n_j])$ are effective by Lemma~\ref{lem:explicit}. Thus $0<[A^+]\ldot\npf{\phi}([p_i])=\phi^*([A^+])\ldot[p_i]$ and $0<[A^+]\ldot\npf{\phi}([n_j])=\phi^*([A^+])\ldot[n_j]$ for all $i=1,\ldots,k$ and $j=1,\ldots,m-1$.
In particular, $\eps'>0$ by the previous consideration. With this definition, Assumption~\eqref{prop:assumption2} yields that 
\begin{equation}\label{eq:negative3}
\bigl([K_X]+\eps'\phi^*([A^+])\bigr)\ldot\eta\leq([K_X]+\eps[A])\ldot\eta\text{ for all }\eta\in\mor{X}.
\end{equation}
Note that it is sufficient to check inequality \eqref{eq:negative3} for the classes $[p_i]$ and $[n_j]$ since $\mor{X}=\langle[p_1],\ldots,[p_k],[n_1],\ldots,[n_m]\rangle_{\R_+}$.
Now set \[\eps^+:=\tfrac{\eps'}{2}.\]

\emph{Claim.} Every class $\gamma\in\nic{X^+}\setminus\{0\}$ is $([K_{X^+}]+\eps^+[A^+])$--negative.

\emph{Proof of the Claim.} Let $0\neq\gamma\in\nic{X^+}$. By definition of $\nic{X^+}$, there exists a sequence of effective cycles $(\gamma_l)_{l\in\N}\subset\nio{X^+}$ such that $\gamma_l\xra{l\to\infty}\gamma$. Let $l\in\N$ be an arbitrary integer. The class $\gamma_l$ is given by a finite sum \[\gamma_l=\sum_{i=1}^sa_i[c^+_i],\] where $a_i\geq 0$ and $c^+_i$ is an irreducible curve on $X^+$ which is not contained in the set $P=\{n_m^+,p_1^+,\ldots,p_k^+\}$ for all $i=1,\ldots,s$. Lemma~\ref{lem:explicit} yields that $\npb{\phi}([c^+_i])=[c_i]+k_i[n_m]$, where $c_i$ denotes the strict transform of $c_i^+$ and $k_i$ is a suitable non-negative integer, for all $i=1,\ldots,s$. 

Note that each $c_i\neq p_j$, for all $j=1,\ldots,k$ since none of the curves $c_i^+$ is equal to one of the $p_j^+$, $j=1,\ldots,k$.

Therefore, $([K_X]+\eps[A])\ldot([c_i]+k_i[n_m])<0$ for all $i=1,\ldots,s$ and \begin{equation}\label{eq:negative4}\npb{\phi}(\gamma_l)\text{ is an effective }([K_X]+\eps[A])\text{--negative class.}\end{equation}
Thus
\begin{align*}
([K_{X^+}]+\eps'[A^+])\ldot\gamma_l&=\phi^*([K_{X^+}]+\eps'[A^+])\ldot\npb{\phi}(\gamma_l)&&\text{ }\\
&=([K_X]+\eps'\phi^*([A^+]))\ldot\npb{\phi}(\gamma_l)&&\text{ }\\
&\leq([K_X]+\eps[A])\ldot\npb{\phi}(\gamma_l)<0&&\text{ by }\eqref{eq:negative3}\text{ and }\eqref{eq:negative4}
\end{align*}

This yields that $([K_{X^+}]+\eps'[A^+])\ldot\gamma=\lim\limits_{l\to\infty}([K_{X^+}]+\eps'[A^+])\ldot\gamma_l\leq0$.
Since $\gamma\in\nic{X^+}\subset\mor{X^+}$ and $A^+$ is ample, we obtain that

\hspace*{1cm}$[K_{X^+}]+\eps^+[A^+])\ldot\gamma=\underbrace{([K_{X^+}]+\eps'[A^+])\ldot\gamma}_{\leq0}-\underbrace{\tfrac{\eps'}{2}[A^+]\ldot\gamma}_{>0}<0$.\hfill$\square_{\text{Claim}}$

By definition of the cone $\nic{X^+}$, this yields that every irreducible curve $c^+\neq p_1^+,\ldots,p_k^+$ which is not contained in $S^+$ is $(K_{X^+}+\eps^+A^+)$--negative.

Now let $R:=\R_+\nu$ be an extremal ray of $\mor{X^+}$ such that $\nu$ is not numerically proportional to $[p_1^+],\ldots,[p_k^+]$ or $[n_m^+]$. The existence of such an extremal ray is guaranteed by the fact that $\mor{X^+}_{K_{X^+}<0}\neq\emptyset$. The given decomposition of $\mor{X^+}$ yields that $R\subset\nic{X^+}$ and thus $([K_{X^+}]+\eps^+[A^+])$--negative. 

Since there are only finitely many $([K_{X^+}]+\eps^+[A^+])$--negative extremal rays in $\mor{X^+}$ by Mori's Cone Theorem, this shows that $\mor{X^+}$ has only finitely many extremal rays and concludes the proof.
\end{proof}
\begin{bem}\label{rmk:inductionstep}
The proof of Proposition~\ref{prop:inductionstep} shows that \[\mor{X^+}=\nic{X^+}+\R_+[n_m^+]+\R_+[p_1^+]+\dots+\R_+[p_k^+],\] where $\nic{X^+}\subset\mor{X^+}_{K_{X^+}<0}$, $n_m^+$ is a curve in the exceptional locus of the flipped small contraction and the $p_i^+$ are the strict transforms of the $K_X$--non-negative curves $p_i$, $i=1,\ldots,k$, under the flip $\phi$. Moreover, every irreducible curve which is not numerically proportional to $p_1^+,\ldots,p_k^+$ or $n_m^+$, is contained in $\nic{X^+}$.
Note that it is not true that the strict transforms of curves which span extremal rays of $\mor{X}$ will always span extremal rays of $\mor{X^+}$. Even if they do, the types of extremal contractions can change.
\end{bem}
\begin{defn}
A smooth projective fourfold which satisfies the requirements of Proposition~\ref{prop:inductionstep} is called a \emph{\polycone}. Let $X_0$ be a \polycone. A finite sequence
\[X_0\stackrel{\phi_1}{\bir}X_1\stackrel{\phi_2}{\bir}\cdots\stackrel{\phi_n}{\bir}X_n\]
of birational maps is called a \emph{\pmc\ for} $X_0$ if
\begin{enumerate}[(i)]
\item the map $\phi_i:X_{i-1}\bir X_i$ is the flip of a small contraction which contracts a $K_{X_{i-1}}$--negative extremal ray $\R_+[s_{i-1}]$, for all $i=1,\ldots,n$, and
\item the Mori cone $\mor{X_n}$ of the fourfold $X_n$ has no $K_{X_n}$--negative small extremal rays.
\end{enumerate}
We will call a \pmc\ for $X_0$ a \emph{\pmc\ for} $\R_+[s_0]$ if the first map $\phi_1:X_0\bir X_1$ in the sequence is the flip of the small contraction which contracts the $K_{X_0}$--negative extremal ray $\R_+[s_0]$.

The number of flips in a \pmc\ is called the \emph{length of the \pmc}.
\end{defn}
\begin{bem}
Note the following two apparent statements.
\begin{enumerate}[(i)]
\item Each variety $X_i$, for $i=0,\ldots,n-1$, in such a \pmc\ is a \polycone\ by Proposition~\ref{prop:inductionstep}. In particular, $\mor{X_n}$ is polyhedral.
\item If $X$ is a smooth Fano fourfold such that $\mor{X}$ has a small extremal ray, then $X$ is a \polycone. This is obvious, as we can take $A=-K_X$ and $\eps=\tfrac{1}{2}$, for example.
\end{enumerate}
\end{bem}
We will now show that \pmc s exist and that there are just finitely many \pmc s for each \polycone. This is an immediate consequence of the following theorem due to Y. Kawamata, K. Matsuda and K. Matsuki.
\begin{thm}[See \protect{\cite[Theorem~5-1-15]{kmm}}]\label{thm:termination}
There is no infinite sequence of flips for threefolds and fourfolds.\qed
\end{thm}
Thanks to this result we are able to prove the afore-noted statement.
\begin{lemma}\label{lem:existence}
Let $X_0$ be a \polycone. Then there exists a \pmc\ for every small extremal ray of $\mor{X_0}$. Moreover, there exist only finitely many \pmc s for $X_0$.
\end{lemma}
\begin{proof}
The existence of a \pmc\ for every small contraction is an immediate consequence of Corollary~\ref{cor:kawamata}, Proposition~\ref{prop:inductionstep} and Theorem~\ref{thm:termination}.

We will prove the second statement of the lemma by contradiction. Suppose there are infinitely many \pmc s for $X_0$. Since $\mor{X_0}$ has only finitely many extremal rays, there has to be a small extremal ray $\R_+[s_0]$ with flip $\phi_1:X_0\bir X_1$ such that infinitely many \pmc s start with $\phi_1$. Proposition~\ref{prop:inductionstep} yields that $\mor{X_1}$ is polyhedral, too. Therefore, there has to be a small extremal ray $\R_+[s_1]$ of $\mor{X_1}$ with flip $\phi_2:X_1\bir X_2$ such that infinitely many \pmc s start with the map $\phi_1\circ\phi_2$. In this manner we can construct an infinite sequence of flips successively which, however, is a contradiction to Theorem~\ref{thm:termination}.
\end{proof}
\begin{constr}
Let $X$ be a \polycone\ and let $\R_+[s]$ be a small extremal ray of $\mor{X}$. Let $X_k$ be a smooth projective fourfold which appears in a \pmc\ for $\R_+[s]$ and denote by $\Phi_k:X\bir X_k$ the induced birational map. The Mori cone of $X_k$ is polyhedral,  \[\mor{X_k}=\{\alpha\in\Nc{X_k}\mid \alpha\ldot[N_i]\geq0,i=1,\ldots,m\}\text{ say,}\] where $N_1,\ldots,N_m$ are nef divisors on $X_k$ which span $\nef{X_k}$.

We define \[\eq{X_k}_{nef}:=\{(\Phi_k)^*([N_i])\mid i=1\ldots,m\}\] as the set of pullbacks via the map $\Phi_k$ of nef divisors on $X_k$ which span $\nef{X_k}$. 
We set \[\eq{X_k}_{div}:=\{(\Phi_k)^*([E_i])\mid i=1,\ldots,l\},\] where $E_1,\ldots,E_l$ are the exceptional divisors which correspond to  the divisorial extremal rays of $\mor{X_k}$.

Furthermore, let $\textup{Poly}(\R_+[s])$ be the set of varieties which appear in a \pmc\ for $\R_+[s]$. 

We call the set \[\eq{\R_+[s]}:=\bigcup\limits_{X_k\in\textup{Poly}(\R_+[s])}(\eq{X_k}_{nef}\cup\eq{X_k}_{div})\] the \emph{set of equations for} $\R_+[s]$.
\end{constr}
\begin{bem}\label{rmk:finite}
Note that the set $\eq{\R_+[s]}$ is a finite set of classes of divisors on $X$ for every small extremal ray of $\mor{X}$.
\end{bem}
\section{Proof of the main theorem}
We will now give the definition of the set $\eq{X}$, which was introduced in the statement of the main theorem of this paper.
\begin{defn}
If $X$ is a smooth Fano fourfold, then set 
\[\eq{X}:=\big(\bigcup\limits_{i=1}^k\eq{\R_+[s_i]}\big)\cup\eq{X}_{nef}\cup\eq{X}_{div},\]
where $\R_+[s_1],\ldots,\R_+[s_k]$ are the small extremal rays of $\mor{X}$. We call $\eq{X}$ the \emph{set of equations for} $X$.
\end{defn}

Note that $\eq{X}$ is a finite set of classes of divisors by Remark~\ref{rmk:finite}.
\subsection{Proof of Theorem~\ref{thm:mainthm}}
Let $X$ be a smooth Fano fourfold. The inclusion \[\mov{X}\subseteq\{\gamma\in\Nc{X}\mid \gamma\ldot\Delta\geq0\text{ for all }\Delta\in\eq{X}\}=:M\] follows from the fact that the numerical pushforward of a movable class by a flip is again a movable class. Let us prove that $\mov{X}\supseteq M$.

Let $\gamma\in\Nc{X}$ such that $\gamma\ldot\Delta\geq0$ for all $\Delta\in\eq{X}$ and let $D$ be an arbitrary irreducible divisor on $X$. We need to show that $\gamma\ldot[D]\geq0$ by Theorem~\ref{thm:eq}. If $[D]$ is contained in the set $\langle\eq{X}\rangle_{\R_+}$ of effective linear combinations of classes in $\eq{X}$, then there is nothing to show. Thus we can assume that $[D]$ is not contained in $\langle\eq{X}\rangle_{\R_+}$.

The inclusion $\eq{X}_{nef}\subset\eq{X}$ yields that $\gamma\in\mor{X}$. Since $X$ is Fano, we know that \[\mor{X}=\langle[f_1],\ldots,[f_m],[d_1],\ldots,[d_n],[s_1],\ldots,[s_k]\rangle_{\R_+},\]
where $f_i,d_j,s_l$ are rational curves on $X$, $\R_+[f_i]$ is an extremal ray of fibre type, $\R_+[d_j]$ is a divisorial extremal ray and $\R_+[s_l]$ is a small extremal ray, for all $i=1,\ldots,m$, $j=1,\ldots,n$ and $l=1,\ldots,k$.

Thus it is sufficient to show that $[f_i]\ldot[D]\geq0$, $[d_j]\ldot[D]\geq0$ and $[s_l]\ldot[D]\geq0$, for all $i=1,\ldots,m$, $j=1,\ldots,n$ and $l=1,\ldots,k$.

We choose arbitrary indices $1\leq i\leq m$, $1\leq j\leq n$ and $1\leq l\leq k$. 

By assumption, $\R_+[f_i]$ is an extremal ray of fibre type. Therefore, clearly $[f_i]\ldot[D]\geq0$ since $D$ cannot contain every fibre of the corresponding contraction.

Let $E_j$ denote the exceptional divisor of the extremal contraction corresponding to $\R_+[d_j]$. Since $[D]\notin\eq{X}_{div}$, we have $[D]\neq[E_j]$. Therefore, we can find a curve $c\subset E_j$ such that $c\nprop d_j$ and $c\nsubseteq D$. Hence $[d_j]\ldot[D]\geq0$.

The last inequality $[s_l]\ldot[D]\geq0$ is shown in the following Proposition~\ref{prop:positive}.
\qed
\begin{prop}\label{prop:positive}
Let $X_0$ be a smooth Fano fourfold and let $\R_+[s_0]$ be a small extremal ray of $\mor{X_0}$. If $D$ is an irreducible divisor on $X_0$ such that $[D]$ is not contained in the closed cone $\langle\eq{\R_+[s_0]}\rangle_{\R_+}$ spanned by classes in $\eq{\R_+[s_0]}$, then $[D]\ldot[s_0]\geq0$.
\end{prop}

Since the proof of Proposition~\ref{prop:positive} is rather long,
we have subdivided it into a number of steps for the reader's convenience.

\subsection{Proof of Proposition~\ref{prop:positive}}

\subsection*{Step 1, setup of notation}
Let $D$ be an irreducible divisor on $X_0$ such that $[D]$ is not contained in the closed cone $\langle\eq{\R_+[s_0]}\rangle_{\R_+}$ spanned by classes in $\eq{\R_+[s_0]}$. Furthermore, let 
\[
  \xymatrix{
      X_0 \ar@{-->}[rr]^{\phi_1} \ar[rd]_{\ph_0}  &     &  X_1 \ar[dl]^{\ph^+_0}  \\
                             &  Y_0  &
  }
\]
be the flip diagram for $\R_+[s_0]$. Following the notation introduced in Figure~\ref{fig:xx}, let $[s_0^1]$ be the class of an irreducible curve $s_0^1$ in a fibre of $\ph_0^+$ and let $D_1$ be the strict transform of $D$ under $\phi_1$. Note that $-s_0\ldot K_{X_0}=1=s_0^1\ldot K_{X_1}$ by Equation~\eqref{eq:number_one} of page~\pageref{eq:number_one}. We may assume that the exceptional locus of $\ph_0$ is irreducible. Then Remark~\ref{rmk:inductionstep} yields that \[\mor{X_1}=\nic{X_1}+\R_+[s_0^1],\] where $\nic{X_1}$ is $K_{X_1}$--negative, and that $s_0^1$ is the only $K_{X_1}$--non-negative curve on $X_1$. 

We will prove Proposition~\ref{prop:positive} by contradiction. Assume that $[D]\ldot[s_0]<0$. 
We know that $[D_1]=(\phi_1)_*([D])$ and Lemma~\ref{lem:explicit} gives $(\phi_1)_*([D])\ldot[s_0^1]>0$.

We will now construct an infinite sequence of flips successively, which is
impossible by Theorem~\ref{thm:termination}, and therefore a
contradiction. Throughout the proof of
  Proposition~\ref{prop:positive}, we use the notation outlined in
  Figure~\ref{fig:xx}.
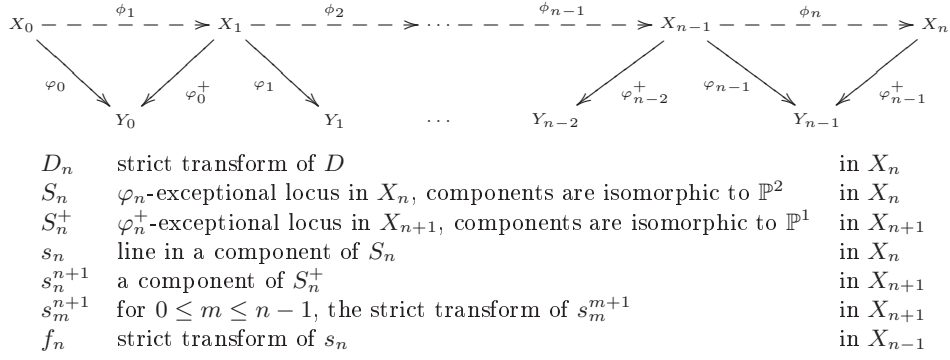
\begin{figure}
  \centering
  \begin{tiny}
    \[
    \xymatrix{
      X_0 \ar@{-->}[rr]^{\phi_1} \ar[rd]_{\ph_0}  &     &  X_1 \ar@{-->}[rr]^{\phi_2} \ar[dl]^{\ph^+_0} \ar[dr]_{\ph_1} &  & \ldots  \ar@{-->}[rr]^{\phi_{n-1}} & & X_{n-1}\ar@{-->}[rr]^{\phi_n} \ar[rd]_{\ph_{n-1}} \ar[ld]^{\ph^+_{n-2}} &  & X_n\ar[dl]^{\ph^+_{n-1}} \\
      &  Y_0  & & Y_1 & \cdots & Y_{n-2} & & Y_{n-1}
    }
    \]
  \end{tiny}
  \begin{small}
  	\begin{tabular}{lll}
    	$D_n$ & strict transform of $D$ & in $X_n$\\
    	$S_n$ & $\varphi_n$-exceptional locus in $X_n$, components are isomorphic to $\mathbb P^2$ & in $X_n$\\
    	$S^+_n$ & $\varphi^+_n$-exceptional locus in $X_{n+1}$, components are isomorphic to $\mathbb P^1$ & in $X_{n+1}$\\
    	$s_n$ & line in a component of $S_n$ & in $X_n$\\
    	$s_n^{n+1}$ & a component of $S^+_n$ & in $X_{n+1}$\\
    	$s_m^{n+1}$ & for $0\leq m\leq n-1$, the strict transform of $s_m^{m+1}$ & in  $X_{n+1}$\\
    	$f_n$ & strict transform of $s_n$ & in $X_{n-1}$
  	\end{tabular}
  \end{small}
  \caption{Notation for sequences of flips}
  \label{fig:xx}
\end{figure}

\subsection*{Step 2, finding a flip sequence of length two}

By Lemma~\ref{lem:existence} there exists a \pmc\ for $\R_+[s_0]$ and, of course, $X_1$ appears in every \pmc\ for $\R_+[s_0]$. Thus $[D_1]$ cannot be a nef class since $[D]\notin\langle\eq{\R_+[s_0]}\rangle_{\R_+}$ by assumption. Hence there has to be a geometrically extremal ray $R_1$ of $\mor{X_1}$ such that $[D_1]$ is negative on $R_1$. Proposition~\ref{prop:inductionstep} yields that $R_1$ is $K_{X_1}$--negative since $[D_1]\ldot[s_0^1]>0$. The contraction of the ray $R_1$ cannot be a fibre contraction since $[D_1]$ is effective and since $[D]\notin\langle\eq{\R_+[s_0]}\rangle_{\R_+}$, it cannot be a divisorial contraction. Therefore, $R_1=\R_+[s_1]$ is a small extremal ray with exceptional locus $S_1$. Without loss of generality, we may assume that $S_1$ is irreducible and hence $S_1\iso\pro^2$.
Let 
\[
  \xymatrix{
      X_1 \ar@{-->}[rr]^{\phi_2} \ar[rd]_{\ph_1}  &     &  X_2 \ar[dl]^{\ph^+_1}  \\
                             &  Y_1  &
  }
\]
be the flip diagram for $\R_+[s_1]$. With the notation of Figure~\ref{fig:xx}, Remark~\ref{rmk:inductionstep} yields that \[\mor{X_2}=\nic{X_2}+\R_+[s_0^2]+\R_+[s_1^2],\] where $\nic{X_2}$ is entirely $K_{X_2}$--negative and $s_0^2,s_1^2$ are the only $K_{X_2}$--non-negative curves on $X_2$. Lemma~\ref{lem:explicit} yields that \[[D_2]\ldot[s_1^2]=(\phi_2)_*([D_1])\ldot[s_1^2]=[D_1]\ldot\npb{(\phi_2)}([s_1^2])=-[D_1]\ldot[s_1]>0.\]

\begin{claim}\label{cl:l2}$\mor{X_2}$ has a small $K_{X_2}$--negative extremal ray $R_2$.\end{claim}
\begin{cproof}[\ref{cl:l2}] Assume that there is no small $K_{X_2}$--negative extremal ray in $\mor{X_2}$. Then, by definition, $\phi_2\circ\phi_1$ is a \pmc~for $\R_+[s_0]$. Since $[D]\notin\langle\eq{\R_+[s_0]}\rangle_{\R_+}$, the divisor class $[D_2]=(\phi_2)_*\bigl((\phi_1)_*([D])\bigr)$ is not nef and there exists a geometrically extremal ray $R$ of $\mor{X_2}$ such that $D_2$ is negative on $R$. The intersection $K_{X_2}\ldot R$ cannot be negative. If it was, then the contraction of the ray $R$ would be either a divisorial or a fibre contraction. However, it cannot be a fibre contraction since $D_2$ is effective and it cannot be a divisorial contraction since $[D]\notin\langle\eq{\R_+[s_0]}\rangle_{\R_+}$. The ray $R$ is therefore $K_{X_2}$--non-negative and Proposition~\ref{prop:inductionstep} shows that $R$ is either the ray $\R_+[s_1^2]$ or the ray $\R_+[s_0^2]$. However, we have already seen that $[D_2]\ldot[s_1^2]>0$. Thus, our assumption yields that $[D_2]\ldot[s_0^2]<0$. To conclude our argumentation, we will show that in fact $[D_2]\ldot[s_0^2]>0$, which will yield the desired contradiction.

If $s_0^1$ and $S_1$ are disjoint, then $\npb{(\phi_2)}([s_0^2])=[s_0^1]$ by Lemma~\ref{lem:explicit} and we obtain that \begin{equation}\label{eq:pos2}[D_2]\ldot[s_0^2]=(\phi_2)^*([D_2])\ldot\npb{(\phi_2)}([s_0^2])=[D_1]\ldot[s_0^1]>0.\end{equation}
Moreover, an analogous computation shows that $[K_{X_2}]\ldot[s_0^2]=1$ in this case.

To prove Claim~\ref{cl:l2} it is thus sufficient to show that $s_0^1$ does not intersect the surface $S_1\iso\pro^2$. To see this, we will investigate the numerical pullback of $[s_1]$, which is the class of an arbitrary line in the surface $S_1\iso\pro^2$, via $\phi_1$ on $X_0$. Lemma~\ref{lem:explicit} gives \[\npb{(\phi_1)}([s_1])=[f_1]+k[s_0],\] where $k\geq0$ with equality iff $s_1$ and $s_0^1$ are disjoint. See Figure~\ref{fig:xy}.
\begin{figure}
	\centering
	\includegraphics[width=12cm]{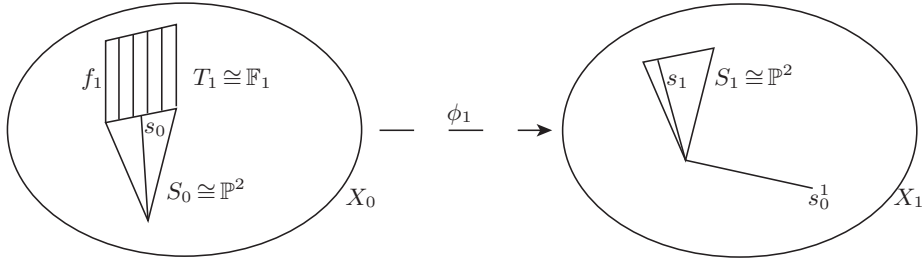}
		\begin{picture}(0,0)%
			\put(-13.0,20.0){\small{$X_1$}}
			\put(-218.0,20.0){\small{$X_0$}}
			\put(-80.0,65.0){\small{$S_1\iso\pro^2$}}
			\put(-45.0,20.0){\small{$s_0^1$}}
			\put(-98.0,65.0){\small{$s_1$}}
			\put(-275.0,65.0){\small{$T_1\iso\mathbb{F}_1$}}
			\put(-285.0,22.0){\small{$S_0\iso\pro^2$}}
			\put(-293.5,47.0){\small{$s_0$}}
			\put(-317.0,65.0){\small{$f_1$}}
			\put(-180.0,52.0){\small{$\phi_1$}}
		\end{picture}
	\caption{Sketch for the proof of Claim~\ref{cl:l2}.}
	\label{fig:xy}
\end{figure}
Equation~\eqref{eq:number_one} yields that
\[-1=[K_{X_1}]\ldot[s_1]=(\phi_1)^*([K_{X_1}])\ldot\npb{(\phi_1)}([s_1])=[K_{X_0}]\ldot([f_1]+k[s_0])=[K_{X_0}]\ldot[f_1]-k\]
Since $X_0$ is Fano, we obtain $0>[K_{X_0}]\ldot[f_1]=k-1\Ra k=0$. This yields that $s_0^1$ is disjoint from any line in $S_1\iso\pro^2$. Thus $s_0^1$ is disjoint from $S_1$ and hence $[D_2]\ldot[s_0^2]>0$ by Equation~\eqref{eq:pos2}. In particular, the curves $s_0^2$ and $s_1^2$ are disjoint. Claim~\ref{cl:l2} is therefore shown.
\end{cproof}\hfill $\square_{\text{Claim~\ref{cl:l2}}}$

So far we have shown the following.
\begin{summ}
The curves $s_0^2$ and $s_1^2$ are disjoint and they are the only irreducible curves in $X_2$ with $[s_0^2],[s_1^2]\in\mor{X_2}_{K_{X_2}\geq0}$. Moreover, we know that $\mor{X_2}=\nic{X_2}+\R_+[s_0^2]+\R_+[s_1^2]$, that $[s_0^2]\ldot[K_{X_2}]=1=[s_1^2]\ldot[K_{X_2}]$ and that $[s_0^2]\ldot[D_2]>0,[s_1^2]\ldot[D_2]>0$. The divisor $D_2$ is negative on a small $K_{X_2}$--negative extremal ray $R_2=\R_+[s_2]$ of $\mor{X_2}$ and $\nic{X_2}$ is entirely $K_{X_2}$--negative.
\end{summ}

\subsection*{Step 3, extending an existing flip sequence}

Now assume that we have constructed a sequence of flips of length $n$, as in Figure~\ref{fig:xx}, with the following properties
\begin{enumerate}[(A$_n$)]
  \item\label{A} The $X_1,\ldots,X_n$ are \polycone s
    and \[\mor{X_n}=\nic{X_n}+\R_+[s_0^n]+\ldots+\R_+[s_{n-1}^n],\] where
    $\nic{X_n}\subset\left(K_{X_n}\right)_{<0}$, the curves $s_0^n,\ldots,s_{n-1}^n$ are pairwise disjoint and are the only
    $K_{X_n}$--non-negative curves in $X_n$.

  \item\label{B} We have $[K_{X_n}]\ldot[s_i^n]=1$, for $i=0,\ldots,n-1$.

  \item\label{C} The strict transform $D_n\subset X_n$ of $D$ under
    $\phi:=\phi_n\circ\ldots\circ\phi_1$ is positive on all $K_{X_n}$--non-negative
    extremal rays $\R_+[s_i^n]$ of $\mor{X_n}$.
    
   \item\label{D} $D_n$ is negative on a small $K_{X_n}$--negative 
   	extremal ray $R_n=\R_+[s_n]$ of $\mor{X_n}$.
\end{enumerate}
We may assume that the exceptional locus $S_n$ of the corresponding extremal contraction is irreducible.

Theorem~\ref{thm:kawamata} and Corollary~\ref{cor:kawamata} yield that
$S_n\iso\pro^2$ and that the flip \[\phi_{n+1}:X_n\bir X_{n+1}\] of $R_n$
exists. Remark~\ref{rmk:inductionstep} yields
that \[\mor{X_{n+1}}=\nic{X_{n+1}}+\R_+[s_0^{n+1}]+\ldots+\R_+[s_n^{n+1}],\]
where $\nic{X_{n+1}}$ is entirely $K_{X_{n+1}}$--negative and the $s_i^{n+1}$ are the only $K_{X_{n+1}}$--non-negative curves on $X_{n+1}$.

In order to apply this construction inductively, and thus construct an infinite flip sequence, it suffices to show that the sequence $X_0 \dasharrow \cdot \dasharrow X_{n+1}$ again satisfies the  properties~(\ref{A}$_{n+1}$)--(\ref{D}$_{n+1}$). It will turn out that they are all corollaries of the following claim.

\begin{claim}\label{cl:positive}
    The curves $s^n_i$, $i=0,\ldots,n-1$, and the surface $S_n$ are disjoint.
\end{claim}
\begin{cproof}[\ref{cl:positive}]
We will prove that $s_{n-1}^n$ and $S_n$ are disjoint first. With the notation of Figure~\ref{fig:xx}, if $s_n\subset S_n\iso\pro^2$ is any line, Lemma~\ref{lem:explicit} gives \[\npb{(\phi_n)}([s_n])=[f_n]+l[s_{n-1}],\] where $l\geq0$ with equality iff $s_{n-1}^n$ and $s_n$ are disjoint. Equation~\eqref{eq:number_one} yields that
\[-1=[K_{X_n}]\ldot[s_n]=(\phi_n)^*([K_{X_n}])\ldot\npb{(\phi_n)}([s_n])=[K_{X_{n-1}}]\ldot([f_n]+l[s_{n-1}])\] and we obtain that \[[K_{X_{n-1}}]\ldot[f_n]=l-1.\]
Note that $[K_{X_{n-1}}]\ldot[f_n]<0$, for otherwise $f_n$ is one of the curves $s_i^{n-1}$, which is impossible since all $s_i^n$ are $K_X$--non-negative by Property~(\ref{B}$_n$) and $[K_{X_n}]\ldot[s_n]=-1$ by Equation~\eqref{eq:number_one}. We obtain that  $0>[K_{X_{n-1}}]\ldot[f_n]=l-1\Ra l=0$ and hence $s_{n-1}^n$ is disjoint from any line in $S_n\iso\pro^2$. It follows that $s_{n-1}^n$ and $S_n$ are disjoint. In particular, the curves $s_{n-1}^{n+1}$ and $s_n^{n+1}$ are disjoint too.

We go on with the curve $s_{n-2}^n$. Note that $s_{n-1}^n$ and $S_n$ are disjoint by the previous step and $s_{n-2}^n$ is disjoint from $s_{n-1}^n$ by Property~(\ref{A}$_n$). This yields that the flip $\phi_n$ is an isomorphism along $S_n$ and $s_{n-2}^n$. Thus, we can investigate the intersection of $s_{n-2}^{n-1}$ with the inverse image $T_n$ of $S_n$ under $\phi_n$, which is again isomorphic to $\pro^2$. This situation is completely analogous to the one of the previous step and the analogous computations yields that $s_{n-2}^{n-1}$ is disjoint form $T_n$. Hence $s_{n-2}^n$ is disjoint form $S_n$ and $s_{n-2}^{n+1}$ is disjoint from $s_n^{n+1}$.

In this manner, we can prove successively that the curves $s^n_i$ are disjoint from $S_n$ and that the curves $s_i^{n+1}$ are disjoint from $s_n^{n+1}$.
\end{cproof}\hfill $\square_{\text{Claim~\ref{cl:positive}}}$
\subsubsection*{Step 3a, proof of Property~(\ref{A}$_{n+1}$)}
We have already seen that \[\mor{X_{n+1}}=\nic{X_{n+1}}+\R_+[s_0^{n+1}]+\ldots+\R_+[s_n^{n+1}],\]
where $\nic{X_{n+1}}$ is entirely $K_{X_{n+1}}$--negative and the $s_i^{n+1}$ are the only $K_{X_{n+1}}$--non-negative curves on $X_{n+1}$. The fact that the curves $s_0^{n+1},\ldots,s_n^{n+1}$ are pairwise disjoint follows from Claim~\ref{cl:positive}.
\subsubsection*{Step 3b, proof of Property~(\ref{B}$_{n+1}$)}
We have \[[K_{X_{n+1}}]\ldot[s_n^{n+1}]=(\phi_{n+1})_*([K_{X_n}])\ldot[s_n^{n+1}]=-[K_{X_n}]\ldot[s_n]=1\] by Equation~\eqref{eq:number_one}.
Now let $0\leq i\leq n-1$ be an integer. Since $s_i^n$ is disjoint from $S_n$ by Claim~\ref{cl:positive}, Lemma~\ref{lem:explicit} yields that $\npb{(\phi_{n+1})}([s_i^{n+1}])=[s_i^n]$ and 
\[[K_{X_{n+1}}]\ldot[s_i^{n+1}]=(\phi_{n+1})_*([K_{X_n}])\ldot[s_i^{n+1}]=[K_{X_n}]\ldot[s_i^n]=1\]
by Property~(\ref{B}$_n$).
\subsubsection*{Step 3c, proof of Property~(\ref{C}$_{n+1}$)}
With the same argumentation as in the previous step, we obtain that
\[[D_{n+1}]\ldot[s_n^{n+1}]=(\phi_{n+1})_*([D_n])\ldot[s_n^{n+1}]=-[D_n]\ldot[s_n]>0,\]
by Lemma~\ref{lem:explicit} and that \[[D_{n+1}]\ldot[s_i^{n+1}]=(\phi_{n+1})_*([D_n])\ldot[s_i^{n+1}]=[D_n]\ldot[s_i^n]>0,\]
by Property~(\ref{C}$_n$) for $i<n$.
\subsubsection*{Step 3d, proof of Property~(\ref{D}$_{n+1}$)}
We claim that there exists a small $K_{X_{n+1}}$--negative extremal ray $R'$ of $\mor{X_{n+1}}$ such that $D_{n+1}$ is negative on $R'$. Since $[D]\notin\langle\eq{\R_+[s_0]}\rangle_{\R_+}$, the divisor class $[D_{n+1}]=(\phi_{n+1})_*((\phi)_*([D]))$ is not nef and there exists a geometrically extremal ray $R'$ of $\mor{X_{n+1}}$ such that $D_{n+1}$ is negative on $R'$. We have already seen Property~(\ref{C}$_{n+1}$), in other words, that $[D_{n+1}]$ is positive on the rays $\R_+[s_0^{n+1}],\ldots,\R_+[s_n^{n+1}]$. Therefore, Property~(\ref{A}$_{n+1}$) shows that $R'$ is contained in $\nic{X_{n+1}}$ and hence $K_{X_{n+1}}$--negative. However, the corresponding extremal contraction cannot be a fibre contraction since $D_{n+1}$ is effective and since $[D]\notin\langle\eq{\R_+[s_0]}\rangle_{\R_+}$, it cannot be a divisorial contraction. Thus the extremal contraction of the ray $R'$ is small and Property~(\ref{D}$_{n+1}$) is therefore proved. 
\subsection*{Step 4, summary, end of proof}
Summarising, we can note the following. Together the properties $[D]\notin\langle\eq{\R_+[s_0]}\rangle_{\R_+}$ and $[D]\ldot[s_0]<0$ yield that there exists a small extremal ray $R_1$ in the Mori cone of the flipped variety $X_1$ such that the strict transform of $D$ is negative on $R_1$. By construction, the class of the strict transform  of $D$ is not contained in the set $\langle\eq{R_1}\rangle_{\R_+}$ and we can iterate this process. By induction, this construction yields an infinite sequence flips and hence the desired contradiction. This concludes the proof of Proposition~\ref{prop:positive}.
\qed

\subsection{Proof of Proposition~\ref{prop:mainprop}}

The proof is completely analogous to the proof of Theorem~\ref{thm:mainthm}. 
The inclusion \[\mov{X}\subseteq\{\gamma\in\Nc{X}\mid \gamma\ldot\Delta\geq0\text{ for all }\Delta\in\left\{[E_1],\ldots,[E_k]\right\}\cup\nef{X}\}\] can be seen directly, or deduced by Theorem~\ref{thm:eq}. To prove the other inclusion, recall the fact that the Mori cone of a smooth threefold has no small extremal rays. Now we can conclude the proof with the same argumentation as in the proof of Theorem~\ref{thm:mainthm}.
\qed
\section{An example for the computation of the moving cone}
A priori, the set $\eq{X}$ defined in Theorem~\ref{thm:mainthm} seems to be very large and cumbersome. We will now sketch an example   which shows that the situation is not that bad and which illustrates the computation of the set $\eq{X}$. For details see \cite[Example~4.34]{diss}.

\begin{bsp}[See \protect{\cite[Example~4.34]{diss}}]\label{bsp:flipsequence}
Let $\pi':\sE\ra\pro^2$ be the vector bundle $\cO_{\pro^2}\oplus(\cO_{\pro^2}(1)^{\oplus 2})$ over $\pro^2$. The variety \[Y:=\pro(\sE)\] is a smooth fourfold and the induced projection map $\pi:Y\ra\pro^2$ has a section corresponding to the quotient $\sE\ra\cO_{\pro^2}\ra0$. Let $S'\subset Y$ be the image of this section, let $F'$ be a fibre of the projection $\pi$ and let $l$ be a line in $F'$ which does not intersect the plane $S'$. 
We denote by $\Gamma'$ the class of the pullback of a hyperplane in $\pro^2$ and by $\Lambda'$ the class of the line bundle $\cO_{Y}(1)$. Moreover, we denote by $\gamma'$ the class of a line in $S'$ and by $\lambda'$ the class of a line in a fibre of $\pi$.

Now let $\mu:X\ra Y$ be the blow up of $Y$ in $l$, denote by $E$ the exceptonal divisor of the blow up and denote by $\eta$ the class of a curve in a fibre of $\mu|_E:E\ra l$. We fix some more notation.

Let $F$  and $S$ denote the strict transforms under $\mu$ of $F'$ and $S'$, respectively. The class of the strict transform under $\mu$ of a divisor with class $\Gamma'$ which does not contain $F'$ is given by $\mu^*(\Gamma')=:\Gamma$, of a divisor with class $\Gamma'$ which contains the fibre $F'$ by $\Gamma-[E]$, and of a divisor with class $\Lambda'$ by $\mu^*(\Lambda')-[E]=:\Lambda$. The strict transform of a divisor with class $\Lambda'-\Gamma'$ is given by $\Lambda-\Gamma+[E]$.

The class of the strict transform under $\mu$ of a curve with class $\gamma'$ is given by $\npb{\mu}(\gamma')=:\gamma$, of a curve in $F'$ by $\npb{\mu}(\lambda')-\eta=:\nu$, and of a general curve with class $\lambda'$ by $\npb{\mu}(\lambda')=:\lambda=\nu+\eta$.

\emph{Fact:} The variety $X$ is a smooth Fano fourfold and one can check that \[\mor{X}=\langle\gamma,\nu,\eta\rangle_{\R_+}.\]
The extremal contractions of the extremal rays $\R_+\nu$ and $\R_+\gamma$ are small with exceptional loci $F$ and $S$, respectively.
Moreover, some short computations show that $\mor{X}$ is cut out of $\Nc{X}$ by the nef divisor classes $\Gamma$, $\Lambda$ and $\Lambda+[E]=\mu^*(\Lambda')$.
Thus we have \[\eq{X}_{nef}=\{\Gamma,\Lambda,(\Lambda+[E])\}\text{ and }\eq{X}_{div}=\{[E]\}.\]
Now we will compute the \pmc s for $\R_+\nu$ and $\R_+\gamma$, and we will start with the sequence for $\R_+\nu$. Let $\phi_1:X\bir X_1$ be the flip of the small extremal ray $\R_+\nu$. We set 
$\Gamma_1:=(\phi_1)_*(\Gamma)$, 
$\Lambda_1:=(\phi_1)_*(\Lambda)$ and 
$[E_1]:=(\phi_1)_*([E])$.
Let $\nu_1$ denote the class $-\npf{(\phi_1)}(\nu)$ and let $F_1$ be the indeterminacy locus of $\phi_1$ in $X_1$.

The class of the strict transform under $\phi_1$ of a curve with class $\gamma$ is given by $\npf{(\phi_1)}(\gamma)-\nu_1=:\gamma_1$ if the curve intersects $F$, and by $\gamma_1+\nu_1$ otherwise.

The class of the strict transform under $\phi_1$ of a curve with class $\eta$ is given by $\npf{(\phi_1)}(\eta)-\nu_1=:\eta_1$ if the curve intersects $F$, and by $\eta_1+\nu_1$ otherwise.

The class of the strict transform of a curve with class $\lambda$ is given by $\npf{(\phi_1)}(\lambda)=:\lambda_1$.
Now one proves that \[\mor{X_1}=\langle\gamma_1,\nu_1,\lambda_1\rangle_{\R_+}\] and that the extremal contractions of $\R_+\gamma_1$ and $\R_+\lambda_1$ are of fibre type.
Thus there is only one \pmc\ for $\R_+\nu$, which has length one. Furthermore, $\mor{X_1}$ is cut out of $\Nc{X_1}$ by the hyperplanes $\Gamma_1^{\bot}$, $\Lambda_1^{\bot}$ and $(\Gamma_1-[E_1])^{\bot}$.
We obtain that \[\eq{X_1}_{nef}=\{\Gamma,\Lambda,(\Gamma-[E])\}\text{ and }\eq{X_1}_{div}=\emptyset.\]

We go on with the \pmc s for $\gamma$. Let $\phi_2:X\bir X_2$ be the flip of the small extremal ray $\R_+\gamma$. We set 
$\Gamma_2:=(\phi_2)_*(\Gamma)$, 
$\Lambda_2:=(\phi_2)_*(\Lambda)$ and 
$[E_2]:=(\phi_2)_*([E])$.
Set $\gamma_2:=-\npf{(\phi_2)}(\gamma)$ and let $S_2$ be the indeterminacy locus of $\phi_2$ in $X_2$.

The class of the strict transform under $\phi_2$ of a curve with class $\nu$ is given by $\npf{(\phi_2)}(\nu)-\gamma_2=:\nu_2$ if the curve intersects $S$, and by $\nu_2+\gamma_2$ otherwise.

The class of the strict transform under $\phi_2$ of a curve with class $\lambda$ is given by $\npf{(\phi_2)}(\lambda)-\gamma_2=:\lambda_2$ if the curve intersects $S$, and by $\lambda_2+\gamma_2$ otherwise.

The class of the strict transform of a curve with class $\eta$ is given by $\npf{(\phi_2)}(\eta)=:\eta_2$.
Some short computations show that \[\mor{X_2}=\langle\gamma_2,\nu_2,\eta_2\rangle_{\R_+},\] that the extremal contraction of $\R_+\nu_2$ is of fibre type and that the extremal contraction of $\R_+\eta_2$ is divisorial with exceptional divisor $E_2$. 

Thus there is only one \pmc\ for $\R_+\gamma$, which has length one, as well. Furthermore, $\mor{X_2}$ is cut out of $\Nc{X_2}$ by the hyperplanes $\Lambda_2^{\bot}$, $(\Lambda_2+[E_2])^{\bot}$ and $(\Lambda_2-\Gamma_2+[E_2])^{\bot}$.
This yields that \[\eq{X_2}_{nef}=\{\Lambda,(\Lambda+[E]),(\Lambda-\Gamma+[E])\}\text{ and }\eq{X_2}_{div}=\{[E]\}.\]

Combining all this, we have \[\eq{X}=\{\Lambda,\Gamma,[E],(\Gamma-[E]),(\Lambda+[E]),(\Lambda-\Gamma+[E])\}\] and a short computation gives
\begin{align*}
\mov{X}&=\{\varsigma\in\Nc{X}\mid\varsigma\ldot\Delta\geq 0,\text{ for all }\Delta\in\eq{X}\}\\
&=\langle\lambda,(\lambda+\gamma),(\gamma+\nu)\rangle_{\R_+}.
\end{align*}
The complete situation is sketched in the following picture, where the hatched areas inside the Mori cones illustrate the moving cones of $X$, $X_1$ and $X_2$.
\begin{figure}[htbp]
  \centering
  \includegraphics[width=12.5cm]{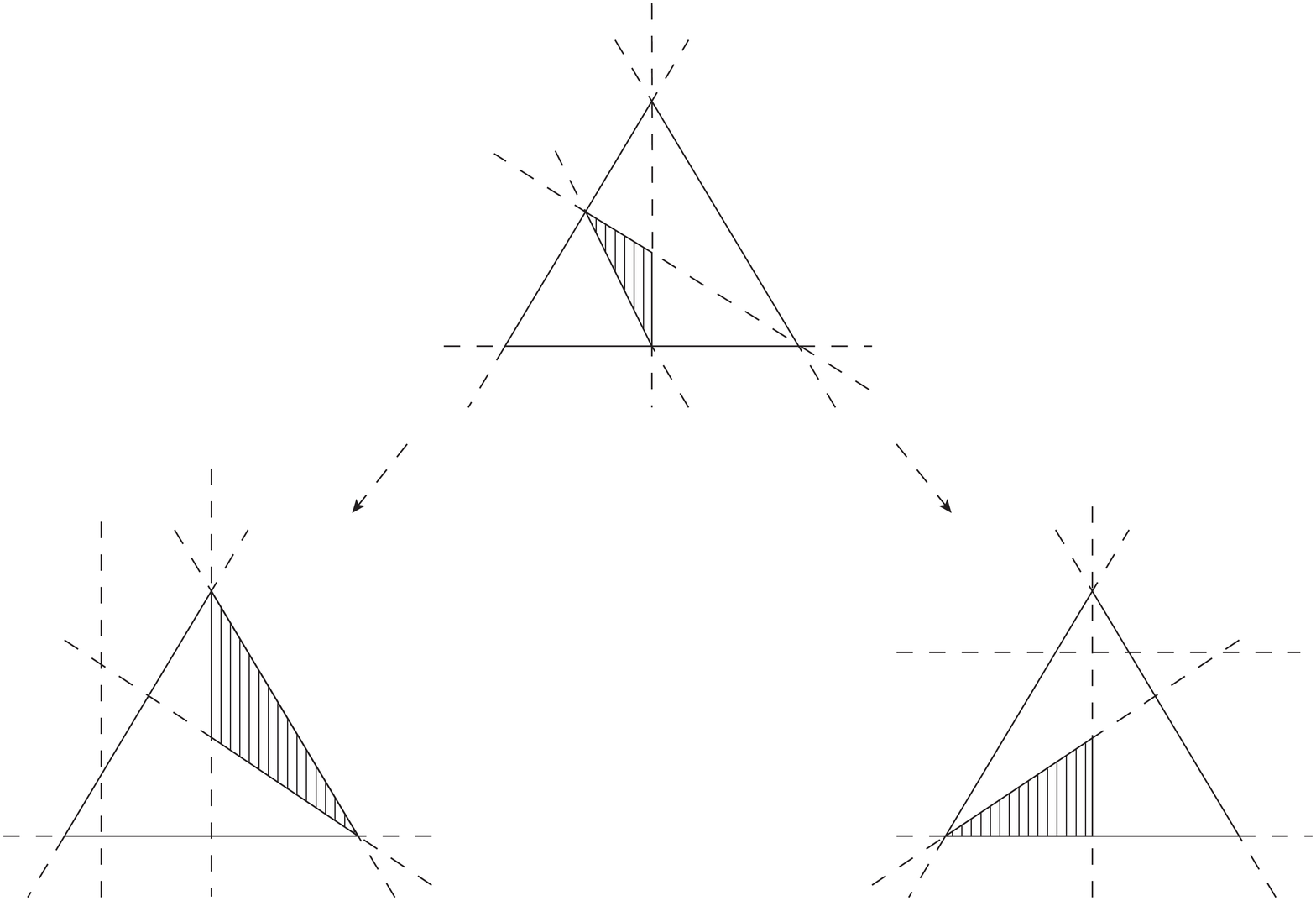}%
		\begin{picture}(0,0)%
				\put(-100.0,240.0){\tiny{cross-sections through}}
				\put(-100.0,230.0){\tiny{the Néron-Severi spaces}}
				\put(-100.0,220.0){\tiny{of $X$,$X_1$ and $X_2$.}}
				\put(-135.0,120.0){\tiny{$\npf{(\phi_2)}$}}
				\put(-247.5,120.0){\tiny{$\npf{(\phi_1)}$}}
				\put(-260.0,240.0){\tiny{$\mor{X}$}}
				\put(-240.0,140.0){\tiny{$\Lambda^{\bot}$}}
				\put(-250.0,152.5){\tiny{$\Gamma^{\bot}$}}
				\put(-180.0,130.0){\tiny{$(\Gamma-[E])^{\bot}$}}
				\put(-198.0,140.0){\tiny{$[E]^{\bot}$}}
				\put(-225.0,237.0){\tiny{$(\Lambda+[E])^{\bot}$}}
				\put(-123.0,140.0){\tiny{$(\Lambda-\Gamma+[E])^{\bot}$}}
				\put(-177.5,216.0){\tiny{$\gamma$}}
				\put(-145.5,145.0){\tiny{$\eta$}}
				\put(-220.0,145.0){\tiny{$\nu$}}
				\put(-365.0,112.0){\tiny{$\mor{X_1}$}}
				\put(-339.0,12.0){\tiny{$\nu_1$}}
				\put(-267.0,12.0){\tiny{$\lambda_1$}}
				\put(-291.0,86.0){\tiny{$\gamma_1$}}
				\put(-355.0,7.0){\tiny{$\Lambda_1^{\bot}$}}
				\put(-255.0,-5.0){\tiny{$(\Gamma_1-[E_1])^{\bot}$}}
				\put(-362.5,22.0){\tiny{$\Gamma_1^{\bot}$}}
				\put(-355.0,77.0){\tiny{$[E_1]^{\bot}$}}
				\put(-322.5,9.0){\tiny{$K_{X_1}^{\bot}$}}
				\put(-322.5,124.5){\tiny{$(\Lambda_1-\Gamma_1+[E_1])^{\bot}$}}
				\put(-31.0,112.0){\tiny{$\mor{X_2}$}}
				\put(-56.0,82.0){\tiny{$\gamma_2$}}
				\put(-99.5,12.0){\tiny{$\nu_2$}}
				\put(-28.5,12.0){\tiny{$\eta_2$}}
				\put(-36.0,74.5){\tiny{$(\Gamma_2-[E_2])^{\bot}$}}
				\put(-131.0,62.5){\tiny{$K_{X_2}^{\bot}$}}
				\put(-58.5,2.0){\tiny{$[E_2]^{\bot}$}}
				\put(-11.0,2.0){\tiny{$(\Lambda_2+[E_2])^{\bot}$}}
				\put(-103.5,2.0){\tiny{$\Lambda_2^{\bot}$}}
				\put(-166.0,22.0){\tiny{$(\Lambda_2-\Gamma_2+[E_2])^{\bot}$}}
		\end{picture}%
  \caption{The hatched areas sketch the moving cones inside the Mori cones.}
  \label{fig:pflipseq4} 
\end{figure}
\end{bsp}
\def\cprime{$'$}

\end{document}